\documentclass[11pt]{article}
\usepackage{amssymb,amsfonts,amsmath,latexsym,epsf,tikz,url}
\usepackage[usenames,dvipsnames]{pstricks}
\usepackage[lined,ruled,linesnumbered]{algorithm2e}
\usepackage{pstricks-add}
\usepackage{epsfig}
\usepackage{pst-grad} 
\usepackage{pst-plot} 
\usepackage[space]{grffile} 
\usepackage{etoolbox} 
\usepackage{float}
\usepackage{soul}
\usepackage{tikz}
\usepackage[colorinlistoftodos]{todonotes}
\usepackage{pgfplots}
\usepackage{mathrsfs}
\usepackage[colorlinks]{hyperref}
\usetikzlibrary{arrows}
\makeatletter 
\patchcmd\Gread@eps{\@inputcheck#1 }{\@inputcheck"#1"\relax}{}{}
\makeatother

\newtheorem{theorem}{Theorem}[section]

\newtheorem{corollary}[theorem]{Corollary}
\newtheorem{lemma}[theorem]{Lemma}

\newtheorem{definition}[theorem]{Definition}

\newcommand{\qed}{\hfill $\square$\medskip}

\begin{document}
	
	\def\nt{\noindent}

	\title{$k$-Fair Coalitions in Graphs}
	
	\author{
		Abbas Jafari  \and   	
		Saeid Alikhani\thanks{corresponding author.}
				}

	
	\maketitle
	
	\begin{center}
		
		Department of Mathematical Sciences, Yazd University, 89195-741, Yazd, Iran	
		\medskip
		{\tt  abbasjafaryb91@gmail.com~~ alikhani@yazd.ac.ir ~~ }
	\end{center}
	
	\begin{abstract}
		Let $G = (V,E)$ be a simple graph. A subset $S \subseteq V$ is called a $k$-fair dominating set if every vertex not in $S$ has exactly $k$ neighbors in $S$. 
			Two disjoint sets $A, B \subseteq V$ form a $k$-fair coalition of $G$ if neither $A$ nor $B$ is a $k$-fair dominating set and  the union $A \cup B$ is a $k$-fair dominating set of $G$. A partition $\pi = \{V_1, V_2, \ldots, V_m\}$ of $V$ is called a $k$-fair coalition partition, if every set $V_i\in\pi$, either  $V_i$ is  a $k$-fair dominating set with exactly $k$ vertices, or $V_i$ is not a $k$-fair dominating set, but forms a $k$-fair coalition with some other set $V_j$ in $\pi$. The $k$-fair coalition number $C_{kf}(G)$ is the largest possible size of a $k$-fair coalition partition for $G$. 
		The objective of this study is to initiate an examination into the notion of $k$-fair coalitions in graphs and present essential findings. 
	\end{abstract}
	

	\medskip
	\noindent{\bf AMS Subj.\ Class.:}  05C69, 05C85.

	\section{Introduction}

	Consider a simple graph $G = (V,E)$, where $V$ represents the vertex set and $E$ the edge set. The {open neighborhood} of a vertex $v$ in $V$ is defined as the set of adjacent vertices, denoted by $N(v)$, while the {closed neighborhood}, represented by $N[v]$, includes~$v$ itself. The { degree} of a vertex $v$, denoted by $deg(v)$, is the number of vertices in its open neighborhood. A vertex $v$ in $G$ is referred to as a {pendant vertex} if it has only one adjacent vertex in its open neighborhood, which is called its {support vertex}. An edge is considered {pendant} if one of its endpoints is a pendant vertex.  In a tree $T$, a vertex with degree one is called a {leaf}, while the adjacent vertex is referred to as the {support vertex}. 
	A vertex with degree $n-1$ in a graph $G$ with $n$ vertices is called a { full} or {universal vertex}, while a vertex with degree $0$ is an { isolate}. The minimum and maximum degrees of $G$ are represented by $\delta(G)$ and $\Delta(G)$, respectively.
	A subset $V_i \subseteq V$ is called {singleton} if it has only one element. 
	
		Within the context of a graph $G$, a subset $S$ of vertices is classified as a {dominating set} if for each vertex in $V \setminus S$, there exists at least one vertex in $S$ that is adjacent to it.
		The domination number $\gamma(G)$ of a
		graph $G$ is the minimum cardinality of a dominating set in $G$. The concept of the
		domination in graphs has been studied extensively and several research
		papers have been published on this topic. For a survey on this area, we
		refer the reader to \cite{10,11,12}.

		For $i \geq 1$, a $i$-fair dominating set ($iFD$-set) in $G$, is a dominating set $D$ such that $|N(v) \cap D|=i$ for every vertex $ v \in V\setminus D$.
	The $i$-fair domination number of $G$,  denoted by $\gamma{if}(G)$, is the minimum cardinality of a $iFD$-set. A $iFD$-set of $G$
	of cardinality $\gamma_{if}(G)$ is called a $\gamma_{if}(G)$-set. A fair dominating set, abbreviated
	FD-set, in $G$ is a $iFD$-set for some integer $k\geq 1$. The fair domination number,
	denoted by $\gamma_f(G)$, of a graph $G$ that is not the empty graph is the minimum
	cardinality of an FD-set in $G$. An FD-set of $G$ of cardinality $\gamma_f(G)$ is called a
$\gamma_f(G)$-set. 
	By convention, if $G=\overline{K_n}$, we define $\gamma_f(G)=n$. By the definition it is easy to see that for any graph $G$ of order $n$, $\gamma(G)\leq \gamma_f(G)\leq n$ and 
	$\gamma_f(G)=n$ if and only if $G=\overline{K_n}$.

Caro, Hansberg and Henning in \cite{Henning} showed that for a disconnected graph $G$ (without isolated vertices) of order $n\geq 3$, $\gamma_f(G)\leq n-2$, and they constructed 
	an infinite family of graphs achieving equality in this bound. 
	They also proved that if $T$ is a tree of order $n\geq 2$, then 
	$\gamma_f(T)\leq \frac{n}{2}$ with equality if and only if $T=T'\circ K_1$ for some tree $T'$.

	A {domatic partition} refers to a partition of a vertices of the graph into  dominating sets. Similarly, a {$k$-fair domatic partition} is a partition into $k$-fair dominating sets. The {$k$-fair domatic number} $d_{kf}(G)$ of a graph $G$ is the  size of a $k$-fair domatic partition with the largest size. On the other hand, the maximum size of a $k$-fair  domatic partition is denoted by $d_{kf}(G)$ and is known as the {$k$-fair domatic number}. The domatic number was first introduced by Cockayne and Hedetniemi in their seminal paper \cite{4}.  Further information on these concepts can be found in authoritative sources such as \cite{6,14,15,16}.

	In the seminal work \cite{7}, the concept of coalitions and coalition partitions was first introduced and subsequently explored in the field of graph theory, as evidenced by notable contributions such as \cite{2,3,8,9,DMHEN}. A { coalition}  in a graph $G$ is defined as the union of two disjoint sets of vertices $U_1$ and $U_2$, such that neither $U_1$ nor $U_2$ individually dominates $G$, but their union does. The sets $U_1$ and $U_2$ are referred to as {coalition partners}. On the other hand, a {coalition partition}  $\Phi=\{U_1, U_2, \ldots, U_k\}$  of $G$ is a partition of its vertex set where each $U_i$ in $\Phi$ is either a one-element dominating set of $G$ or a non-dominating set that forms a coalition with another non-dominating set $U_j \in \Phi$. The {coalition number} $C(G)$ of a graph $G$ is the maximum number of sets that can be present in a  coalition partition  of $G$. The concept of total coalition and independent coalition was introduced and explored in \cite{1,DMGT,DAMHEN}, while the coalition parameter for cubic graphs of order at most $10$ was investigated in \cite{2}.
	
	For every  coalition partition $\Phi$ of a graph $G$, there is a corresponding graph called the {coalition graph} of $G$ with respect to $\Phi$, denoted as $CG(G, \Phi)$. The vertices of this graph correspond one-to-one with the sets of $\Phi$, and two vertices are adjacent in $CG(G, \Phi)$ if and only if their corresponding sets form a coalition. The study of coalition graphs, particularly for paths, cycles, and trees, was conducted in \cite{8}. 
	
	Inspired by the concept of $k$-fair dominating sets \cite{Aust}, we introduce a new framework for coalitions based on fairness.
	
   \medskip
	
	In Section 2, we define the $k$-fair coalition model formally and discuss its fundamental properties. In Section 3, we establish the $k$-fair coalition number for several graph families, particularly focusing on regular graphs and graphs with pendant vertices. In Section 4, we investigate graphs with large $k$-fair coalition numbers. In Section 5, we present some results on the $k$-fair coalition number of cubic graphs. Finally, we conclude the paper in Section 6.

	\section{Introduction to $k$-fair coalition}
	
	In this section, we first state the definition of the $k$-fair  coalition and $k$-fair  coalition partition.
	
\begin{definition}
	Two disjoint sets $A, B \subseteq V$ form a {$k$-fair coalition} if
	neither $A$ nor $B$ is a $k$-fair dominating set, but the union $A \cup B$ is a $k$-fair dominating set. 
	A partition $\pi = \{V_1, V_2, \ldots, V_m\}$ of the vertex set of a graph $G$ is called a {$k$-fair coalition partition} if every set $V_i\in \pi$, either   $V_i$ is itself a $k$-fair dominating set with exactly $k$ vertices, or
	 $V_i$ is not a $k$-fair dominating set by itself, but forms a $k$-fair coalition with some other set $V_j$ in $\pi$ (where $j \neq i$).
		\end{definition}

\begin{definition}
	The {$k$-fair coalition number} of a graph $G$, denoted by $C_{kf}(G)$, is the maximum number of sets in a $k$-fair coalition partition of $G$. That is:
	\[
	C_{kf}(G) = \max \{ |\pi| : \pi \text{ is a $k$-fair coalition partition of } G. \}
	\]
\end{definition}

The following theorem gives a lower bound for the $k$-fair coalition number. 

\begin{theorem}\label{thm:1}
	If $G$ is a connected graph and $k \geq 2$, then $C_{kf}(G) \geq 2d_{kf}(G)$.
\end{theorem}

\begin{proof}
	Assume that $G$ has a $k$-fair domatic partition $\mathcal{C} = \{C_1, C_2, \ldots, C_s\}$ with $d_{kf}(G) = s$. Without loss of generality, assume that the sets $\{C_1, C_2, \ldots, C_{s-1}\}$ are minimal $k$-fair dominating sets. If any set $C_i$ is not minimal, we can find a subset $C'_i \subseteq C_i$ that is a minimal $k$-fair dominating set and add the remaining vertices to the set $C_s$.
		Since $k \geq 2$, no $C_i$ is a singleton set. If we split a minimal $k$-fair dominating set with more than one member into two non-empty subsets, we obtain two sets that are not $k$-fair dominating sets by themselves but form a $k$-fair coalition together. Therefore, we partition each non-singleton set $C_i$ into two subsets $C_{i,1}$ and $C_{i,2}$ that form a $k$-fair coalition. This operation results in a new partition $\mathcal{C}'$.
		Now consider the $k$-fair dominating set $C_s$. If $C_s$ is minimal, we split it into two non-$k$-fair dominating subsets and add them to $\mathcal{C}'$, resulting in a $k$-fair coalition partition of size at least $2s$. Since $s = d_{kf}(G)$, we have $C_{kf}(G) \geq 2d_{kf}(G)$.
	
	If $C_s$ is not minimal, we find a minimal subset $C'_s \subseteq C_s$ and partition it into two non-$k$-fair dominating sets $C'_{s,1}$ and $C'_{s,2}$. We then examine the set $C''_s = C_s \setminus C'_s$:
	\begin{itemize}
		\item If $C''_s$ forms a $k$-fair coalition with any set in $\mathcal{C}'$, then by adding it to $\mathcal{C}'$, we obtain a partition of size at least $2s+1$, meaning $C_{kf}(G) \geq 2d_{kf}(G) + 1$.
		\item Otherwise, by replacing $C'_{s,2}$ with $C'_{s,2} \cup C''_s$ in $\mathcal{C}'$, we obtain a partition of size at least $2s$.
	\end{itemize}
		In both cases, we conclude that $C_{kf}(G) \geq 2d_{kf}(G)$, which completes the proof.\qed
\end{proof}

\medskip
The following corollary is an instant consequence of Theorem \ref{thm:1}. 
\begin{corollary}
	For even $k$, $C_{kf}(G) \geq d_{(k/2)f}(G)$.
\end{corollary}

\begin{lemma}\label{lem:5-2}
	For any graph $G$ and any $k$-fair coalition partition $\mathcal{C}$ of $G$, each set in $\mathcal{C}$ can form a $k$-fair coalition with at most $\Delta(G)-k+2$ other sets in $\mathcal{C}$.
\end{lemma}

\begin{proof}
	Let $v$ be a vertex in graph $G$ and let $S \in \mathcal{C}$ be a set such that $v \in S$. If $S$ is a $k$-fair dominating set, then by definition it does not form a $k$-fair coalition with any other set in $\mathcal{C}$, which confirms the result.
		Now suppose $S$ is not a $k$-fair dominating set. In this case, there exists a vertex $x \notin S$ that is not fairly $k$-dominated by $S$, i.e., $|N(x) \cap S| \neq k$.
		We consider two cases:
	
	\noindent{Case 1:} $|N(x) \cap S| > k$.\\
	In this case, $S$ can only form a $k$-fair coalition with the set that contains $x$.
	
\noindent{Case 2:} $|N(x) \cap S| < k$.\\
	For any set $A \in \mathcal{C}$ that does not contain $x$ and forms a $k$-fair coalition with $S$, a necessary condition is that $A \cup S$ contains exactly $k$ neighbors of $x$.
		To maximize the number of sets that form a $k$-fair coalition with $S$, the set $S$ must contain at most $k-1$ neighbors of $x$. This ensures that the remaining neighbors of $x$ are covered by other coalition partners of $S$. Therefore, in the worst case, $S$ can form a $k$-fair coalition with at most $|N(x)|-(k-1)$ subsets that do not contain $x$, or with the subset containing $x$.
		Thus, the number of $k$-fair coalition partners of $S$ is equal to:
	\[
	1+|N(x)|-(k-1) \leq \Delta(G) - k + 2,
	\]
	where the number $1$ refers to the set containing $x$.
		This  completes the proof.\qed
\end{proof}

\medskip

In the following theorem, we provide an upper bound for the $k$-fair coalition number based on maximum degree and $k$.
\begin{theorem}\label{thm:5-6}
	For any graph $G$ with maximum degree $\Delta(G)$ and any integer $k>\delta(G)$, we have:
	\[
	C_{kf}(G) \leq \Delta(G) - k + 3.
	\]
\end{theorem}

\begin{proof}
	Let $x$ be a vertex of $G$ with  $\deg(x) = \delta(G)$. Also, let $\mathcal{C}$ be a $k$-fair coalition partition of $G$ with size $C_{kf}(G)$. Consider the set $X \in \mathcal{C}$ such that $x \in X$.
	
	\noindent{Case 1:} If $N(x) \subseteq X$, then each set in $\mathcal{C} \setminus \{X\}$ can only form a $k$-fair coalition with $X$. Therefore, by Lemma \ref{lem:5-2} we have:
	\[
	C_{kf}(G) \leq 1 + (\Delta(G) - k + 2) = \Delta(G) - k + 3.
	\]
	
	\noindent{Case 2:} If $N(x) \not\subseteq X$, consider two distinct sets $A \neq X$ and $B \neq X$ from $\mathcal{C}$. If $A$ and $B$ form a $k$-fair coalition, then $A \cup B$ must be a $k$-fair dominating set. However, since $x \notin A \cup B$ and $\deg(x) = \delta(G) < k$, this creates a contradiction because $x$ must have exactly $k$ neighbors in $A \cup B$. Therefore, each set in $\mathcal{C}$ can only form a $k$-fair coalition with $X$. Consequently, by Lemma \ref{lem:5-2} we have:
	\[
	C_{kf}(G) \leq (\Delta(G) - k + 2) + 1 = \Delta(G) - k + 3. 
	\]\qed
\end{proof}

We need the following theorem:

\begin{theorem}{\rm\cite{Aust}}
	\label{lthm}
	For any graph $G$ with $\Delta(G)\geq \delta(G)+1$, $$C_{\delta(G)}(G)\le 2\Delta(G)-2\delta(G)+4.$$
\end{theorem}

Since every $k$-fair dominating partition is a $k$-fair coalition partition, by Theorem \ref{lthm}, we have the following result:

\begin{corollary}
	For any graph \( G \) with \( \Delta(G) \geq \delta(G) + 1 \), we have:
	\[
	C_{\delta(G)f}(G) \leq 2\Delta(G) - 2\delta(G) + 4.
	\]
\end{corollary}

We need the following theorem which is about the $k$-coalition number of $k$-regular graphs. 
	\begin{theorem} {\rm\cite{Aust}}\label{regular34}
		If $G$ is a $k$-regular graph, then $3\leq C_k(G) \leq 4$. 
	\end{theorem}

\begin{lemma}
	If  $G$ is a $k$-regular graph, then 
	\[
	3 \leq C_{kf}(G) \leq 4.
	\]
\end{lemma}
\begin{proof}
	In $k$-regular graphs, every $k$-dominating set is also a $k$-fair dominating set, so by Theorem \ref{regular34} the result follows. \qed
\end{proof}

\section{$k$-fair coalition number of specific graphs}
In this section, we study the $k$-fair coalition number for some specific graphs such as complete graphs, trees, paths $P_n$, cycles $C_n$, $P_n \circ K_l$, and $C_n \circ K_l$. 

We begin with complete graph $K_n$ with $V(K_n)=\{v_1,v_2,\dots, v_n\}$. 
Since $\pi = \{\{v_1, v_2, \ldots, v_{k-1}\}, \{v_k\}, \{v_{k+1}\}, \ldots, \{v_n\}\}$ is a $k$-fair coalition partition for $K_n$, we have the following result:

\begin{lemma}
	For every $2 \leq k \leq n - 1$, $C_{kf}(K_n) = n - k + 2$.
\end{lemma}

	\begin{theorem}\label{comp} {\rm\cite{Babak}}
		Let $s, t$ and $k$ be any positive integers with $t\geq s$ and $k\geq 2$. Then,
		\[
		C_k(K_{s,t})=
		\begin{cases}
		2, & s<k;\\
		4, & s=k;\\
		t-k+3, & k+1\leq s\leq 3k-2;\\
		s+t-4k+4, & s\geq 3k-1.
		\end{cases}
		\]
	\end{theorem}

\begin{theorem}
	Let $K_{s,t}$ be a complete bipartite graph with vertex sets $X = \{v_1, v_2, \ldots, v_s\}$ and $Y = \{v'_1, v'_2, \ldots, v'_t\}$, where $s \leq t$. 
	The $k$-fair coalition number of $K_{s,t}$ is as follows:
	
	\begin{enumerate}
		\item[(i)] If $k < s < 3k - 1$, then $C_{kf}(K_{s,t}) \leq t - k + 3.$
		\item[(ii)] If $3k - 1 \geq 3$, then $C_{kf}(K_{s,t}) \leq s + t - 4k + 4.$
		\item[(iii)] If $k = s = t$, then $C_{kf}(K_{s,t}) = 4.$
		\item[(iv)] If $k = s < t$, then $C_{kf}(K_{s,t}) = 2.$
		\item[(v)] If $k > s$, then $C_{kf}(K_{s,t}) = 2.$
	\end{enumerate}
\end{theorem}

\begin{proof}
	\begin{enumerate}
		\item[(i)] It follows from Theorem \ref{comp}.
		\item[(ii)] It follows from Theorem \ref{comp}.
				\item[(iii)] 
		Let $\Theta = \{S_1, S_2, \ldots \}$ be a $k$-fair coalition partition and suppose $S_1, S_2$ form a $k$-fair coalition. 
		Since $\deg_X(v) = k$ for all $v \in Y$, we conclude that $Y \subseteq (S_1 \cup S_2)$ or $X \subseteq (S_1 \cup S_2)$. 
		Therefore, the partition $\Theta$ can have at most four members. According to the following partition, we have $|\Theta| = 4$:
				\[
		\{\{v_1\}, \{v_2, \ldots, v_s\}, \{v'_1\}, \{v'_2, \ldots, v'_t\}\}
		\]
				Note that if there exists a $k$-fair dominating set $S$ in $\Theta$, then $\Theta = \{X, Y\}$.
		
		\item[(iv)]  
		Let $\Theta = \{S_1, S_2, \ldots\}$ be a $k$-fair coalition partition for the complete bipartite graph $K_{s,t}$. Without loss of generality, assume that $S_1$ and $S_2$ form a $k$-fair coalition. 
		
		If $X \nsubseteq S_1 \cup S_2$, then:
		\begin{align*}
			&\exists v \in X \setminus (S_1 \cup S_2)\\
			& \Rightarrow \forall u \in Y: |N(u) \cap (S_1 \cup S_2)| \leq k - 1 \\
			&\Rightarrow Y \subseteq S_1 \cup S_2 \\
			&\Rightarrow \forall v \in X \setminus (S_1 \cup S_2): |N(v) \cap (S_1 \cup S_2)| \geq k + 1 \\
			&\Rightarrow \forall v \in X \setminus (S_1 \cup S_2): |N(v) \cap (S_1 \cup S_2)| \neq k
		\end{align*}
				This contradicts the assumption that $S_1 \cup S_2$ is a $k$-fair coalition. Therefore,
		\[X \subseteq S_1 \cup S_2.\]
		
		Without loss of generality, assume $X \subseteq S_1$. Then for all $u\in Y$:
		\begin{align*}
			 |N(u) \cap S_1| = |X| = k. 
		\end{align*}
		This implies that $S_1$ is a $k$-fair dominating set, which contradicts the fact that $S_1$ and $S_2$ form a coalition. Hence, we must have:
		\begin{align*}
			X \cap S_1 \neq \emptyset \quad \text{and} \quad X \cap S_2 \neq \emptyset.
		\end{align*}
		Therefore, $|\Theta| = 2$.
		
		If $\Theta$ contains a $k$-fair dominating set, then that set must be $X$, and according to the previous argument, any other $k$-fair coalition partner would have to contain $X$, which is not possible.
		
		\item[(v)] 
		Let $\Theta = \{S_1, S_2, \ldots\}$ be a $k$-fair coalition partition for the complete bipartite graph $K_{s,t}$. Since the degree of each vertex in $Y$ is $s$ and $s < k$ (given $s \leq t$ and $k > s$), it follows that $Y \subseteq S_1 \cup S_2$.
		
		\noindent{Case 1:} $t \neq k$. In this case for all $v \in X$, 
					$|N(v) \cap (S_1 \cup S_2)| \neq k$, so 
			$X \subseteq S_1 \cup S_2$ and so 
			$S_1 \cup S_2 = V(K_{s,t})$, therefore $|\Theta| = 2.$

		\noindent{Case 2:}  $t = k$.
		If $Y\cap S_1 \neq \emptyset$ and $Y \cap S_2 \neq \emptyset$, then $|\Theta| = 2$. Otherwise, assume that $Y \subseteq S_1$, so for all
		 $v \in X$, $|N(v) \cap S_1| = |Y| = k$. 
		This means that $S_1$ is a $k$-fair dominating set, which contradicts the assumption that $S_1$ and $S_2$ form a coalition.
				Note that $\Theta$ cannot contain a $k$-fair dominating set, since any $k$-fair dominating set must contain vertices from $Y$, which would result in more than $k$ vertices.\qed
			\end{enumerate}
\end{proof}

Using Theorem \ref{thm:5-6} and the fact that the minimum degree of vertices in trees is one, the following lemma can be concluded.

\begin{lemma}\label{lem:5-8}
	For any tree $T$, the $2$-fair coalition number satisfies the following inequality:
	\[
	C_{2f}(T) \leq \Delta(T) + 1.
	\]
\end{lemma}

The following corollary gives the $2$-fair coalition number of path. 
\begin{corollary}\label{cor:4-14}
	Let $P_n$ be a path on $n$ vertices. Then the $2$-fair coalition number of $P_n$ is:
	\[
	C_{2f}(P_n) =
	\begin{cases}
		1, & n = 1 \\
		2, & n = 2, 3 \\
		3, & n \geq 4.
	\end{cases}
	\]
\end{corollary}

\begin{proof}
	Suppose that $V(P_n)=\{v_1,v_2,\dots,v_n\}$ such that $v_1, v_n$ are leaves and $v_iv_{i+1}$ are edges of $P_n$ for $1\leq i\leq n-1$. The cases $n = 2, 3$ are trivial. For $n \geq 4$, Lemma \ref{lem:5-8} establishes the upper bound $C_{2f}(P_n) \leq 3$. To show that this bound is tight, consider the partition:
	\[
	\pi = \{\{v_1,v_4,v_5, \ldots, v_n\}, \{v_2\}, \{v_3\}\},
	\]
	which constitutes a valid $2$-fair coalition partition of size $3$, thereby completing the proof.\qed
\end{proof}

\medskip
To obtain  the $2$-fair coalition number of cycle $C_n$, we need the following theorem.

\begin{theorem}{\rm\cite{Aust}}\label{cnc}
	The $2$-coalition number of $C_n$ is:
	\[ 
	C_2(C_n)=\begin{cases}
	
	4,& n  \text{ is even}\\
	3,& n  \text{ is odd}.
	\end{cases}
	\]
\end{theorem}

\begin{theorem}
	The $2$-fair coalition number of cycle $C_n$ is as follows:
	\[
	C_{2f}(C_n) = 
	\begin{cases}
		4, & n \text{ is even,} \\
		3, & n \text{ is odd}
	\end{cases}
	\]
\end{theorem}

\begin{proof}		
	By Theorem \ref{cnc}, we have:
	\[
	C_{2f}(C_n) \leq 
	\begin{cases}
		4, & \text{if } n \text{ is even}, \\
		3, & \text{if } n \text{ is odd}.
	\end{cases}
	\]
	To prove equality, we consider the following partitions:
	
	\noindent{Case 1: If $n$ is even:}
	We place the vertices with even indices into two sets and the vertices with odd indices into two other sets:
	\[
	S_1 = \{v_2, v_6, v_{10}, \ldots\}, \quad S_2 = \{v_4, v_8, v_{12}, \ldots\},
	\]
	\[
	S_3 = \{v_1, v_5, v_9, \ldots\}, \quad S_4 = \{v_3, v_7, v_{11}, \ldots\}.
	\]
	It is straightforward to verify that every vertex of even degree is adjacent to exactly two vertices from the odd sets, and vice versa. Therefore, this forms a valid $2$-fair coalition partition with four sets.
	
	\noindent{Case 2: If $n$ is odd:}
	We define one set containing $n-2$ consecutive vertices and two singleton sets:
	\[
	S_1 = \{v_1, v_2, \ldots, v_{n-2}\}, \quad S_2 = \{v_{n-1}\}, \quad S_3 = \{v_n\}.
	\]
	This construction yields a valid $2$-fair coalition partition of size $3$.
	
		Since the provided partitions achieve the upper bounds given by the corollary, we conclude that:
	\[
	C_{2f}(C_n) = 
	\begin{cases}
		4 & \text{if } n \text{ is even}, \\
		3 & \text{if } n \text{ is odd}.
	\end{cases}
	\]\qed
\end{proof}

\begin{theorem}\label{2fcnp}
	The $2$-fair coalition number for the graph $P_n \circ K_1$ is given by:
	\[
	C_{2f}(P_n \circ K_1) = 
	\begin{cases}
		2 & n = 1, 2, 3, 5 \\
		3 & n = 4 \\
		3 & n \geq 6
	\end{cases}
	\]
\end{theorem}

\begin{proof}
	Let $\mathcal{X} = \{S_1, S_2, \ldots\}$ be a $2$-fair coalition partition of the graph. If all degree-one vertices are not contained within a single set, then only the two sets containing these vertices can form a valid $2$-fair coalition. Therefore, we assume that all degree-one vertices are placed into a single set.
		Without loss of generality, let $S_1$ be the set containing all degree-one vertices. Consequently, all other sets can only form a coalition with $S_1$.
		If $v_1 \notin S_1$, then $S_1$ can only form a coalition with the set containing either vertex $v_1$ or $v_2$, which implies:
	\[
	C_{2f}(P_n \circ K_1) \leq 3.
	\]
	A similar argument holds for the vertex $v_n$ (see the following figure).
		We therefore consider $S_1$ to be the set containing all degree-one vertices along with the vertices $v_1$ and $v_n$.
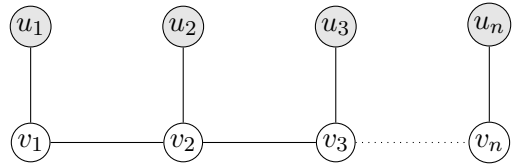
	\begin{center}
		\begin{tikzpicture}[
			node distance=1cm and 1.5cm,
			main node/.style={circle, draw, minimum size=5mm, inner sep=1pt},
			leaf node/.style={circle, draw, fill=gray!20, minimum size=5mm, inner sep=1pt}
			]
			
			\node[main node] (v1) {\( v_1 \)};
			\node[main node, right=of v1] (v2) {\( v_2 \)};
			\node[main node, right=of v2] (v3) {\( v_3 \)};
			\node[main node, right=of v3] (vn) {\( v_n \)};
			
			\node[leaf node, above=of v1] (u1)[above] {\( u_1 \)};
			\node[leaf node, above=of v2] (u2) {\( u_2 \)};
			\node[leaf node, above=of v3] (u3) {\( u_3 \)};
			\node[leaf node, above=of vn] (u4) {\( u_n \)};
			
			\draw (v1) -- (v2);
			\draw (v2) -- (v3);
			\draw [dotted] (v3) -- (vn);
			
			\draw (v1) -- (u1);
			\draw (v2) -- (u2);
			\draw (v3) -- (u3);
			\draw (vn) -- (u4);
								\end{tikzpicture}
		 	\end{center}
	The results are trivial for $n \leq 5$. For $n=4$, we present a $2$-fair coalition partition with three sets in the following figure.
	
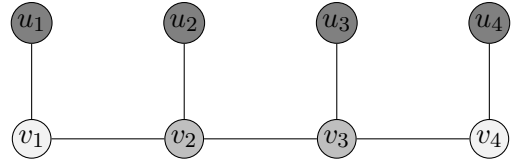
\begin{center}
		\begin{tikzpicture}[
			node distance=1cm and 1.5cm,
			main node/.style={circle, draw, minimum size=5mm, inner sep=1pt},
			leaf node/.style={circle, draw, fill=gray!100, minimum size=5mm, inner sep=1pt}
			]
			
			\node[main node,fill=gray!10] (v1) {\( v_1 \)};
			\node[main node,fill=gray!50, right=of v1] (v2) {\( v_2 \)};
			\node[main node,fill=gray!50, right=of v2] (v3) {\( v_3 \)};
			\node[main node,fill=gray!10, right=of v3] (v4) {\( v_4 \)};
			
			\node[leaf node, above=of v1] (u1)[above] {\( u_1 \)};
			\node[leaf node, above=of v2] (u2) {\( u_2 \)};
			\node[leaf node, above=of v3] (u3) {\( u_3 \)};
			\node[leaf node, above=of v4] (u4) {\( u_4 \)};
			
			\draw (v1) -- (v2);
			\draw (v2) -- (v3);
			\draw (v3) -- (v4);
			
			\draw (v1) -- (u1);
			\draw (v2) -- (u2);
			\draw (v3) -- (u3);
			\draw (v4) -- (u4);
			
		\end{tikzpicture}
		\begin{center}
	\end{center} 
	\end{center}
	
	Since $S_1$ is not a $2$-fair dominating set, there exists a vertex $v$ such that $|N(v) \cap S_1| \neq 2$. If there exists a vertex for
	 which $|N(v) \cap S_1| = 3$, then $S_1$ can only form a coalition with the set containing $v$, implying $C_{2f}(P_n \circ K_1) \leq 2$.
		Now, suppose $|N(v) \cap S_1| = 1$. This implies there exist three consecutive vertices on the path $P_n$ that are not in $S_1$. Let $v_{i-1}, v_i, v_{i+1}$ be the first three such vertices (with $v_{i-2} \in S_1$). The set $S_1$ can only form a coalition with a set containing one of $v_{i-1}$, $v_{i+1}$, or with the set containing $v_i$.
		Let $S'$ be the set containing vertex $v_i$. If it forms a coalition with $S_1$, then it must also contain vertex $v_{i-1}$; otherwise, we would have $|N(v_{i-1}) \cap (S_1 \cup S')| \neq 2$. Consequently, $S_1$ can form a coalition with at most one other set (e.g., one containing $v_{i+1}$). Hence:
	\[
	C_{2f}(P_n \circ K_1) \leq 3.
	\]
	If $S'$ does not form a coalition with $S_1$, the same upper bound of $3$ holds. 	
	To complete the proof, we construct a $2$-fair coalition partition for $n \geq 6$ as follows:
	\begin{align*}
		S_1 &= \{ u_1, u_2, \ldots, u_n \} \cup \{v_1, v_6, v_7, \dots, v_n\}, \\
		S_2 &= \{v_2, v_3\}, \\
		S_3 &= \{v_4, v_5\}.
	\end{align*}
	It can be verified that this partition satisfies the condition of a $2$-fair coalition, which completes the proof.\qed
	
\end{proof}

\medskip

A similar result holds for the $C_n\circ K_1$, which we present now.

\begin{theorem}
	The $2$-fair coalition number of the corona product $C_n \circ K_1$ is given by:
	\[
	C_{2f}(C_n \circ K_1) = 
	\begin{cases}
		4, & \text{if } n = 3; \\
		3, & \text{if } n \geq 4.
	\end{cases}
	\]
\end{theorem}

\begin{proof}
	Using a similar argument to the proof of Theorem \ref{2fcnp}, the vertices of degree one must be in the same set, and for $n \geq 5$ we conclude that:
	\[
	C_{2f}(P_n \circ K_1) \leq 3.
	\]
	To prove equality, we present appropriate $2$-fair coalition partitions for different cases.
	
	\noindent{Case $n = 3$:}
	
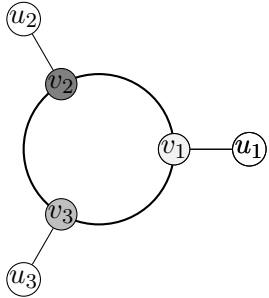
\begin{center}
		\begin{tikzpicture}[scale=1]
			
			\pgfmathsetmacro{\r}{1}
			\pgfmathsetmacro{\d}{1}
			
			\draw[thick] (0,0) circle (\r);
			
			\foreach \angle/\label/\col in {0/1/gray, 120/2/gray, 240/3/gray!50, 0/1/gray!10}
			{
				\node[circle, draw=black, fill=\col, minimum size=0.4cm, inner sep=0] (main-\angle) at (\angle:\r) {$v_\label$};
				
				\node[circle, draw=black, minimum size=0.4cm, inner sep=0] (outer-\angle) at (\angle:\r + \d) {$u_\label$};
				
				\draw (main-\angle) -- (outer-\angle);
			}
			
		\end{tikzpicture}
	\end{center}	
	
	\[
	\mathcal{X} = \left\{ \{u_1, u_2, u_3\}, \{v_1\}, \{v_2\}, \{v_3\} \right\}.
	\]
	
	\noindent{Case $n = 4$:} 
	
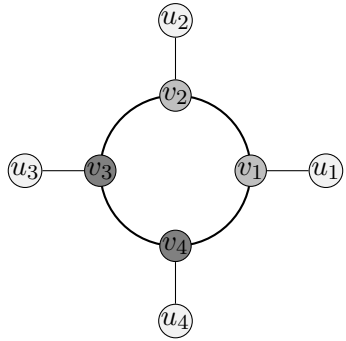
\begin{center}
		\begin{tikzpicture}[scale=1]
			
			\pgfmathsetmacro{\r}{1}
			\pgfmathsetmacro{\d}{1}
			
			\draw[thick] (0,0) circle (\r);
			
			\foreach \angle/\label/\col in {90/2/gray!50, 180/3/gray, 270/4/gray, 0/1/gray!50}
			{
				\node[circle, draw=black, fill=\col, minimum size=0.4cm, inner sep=0] (main-\angle) at (\angle:\r) {$v_\label$};
				
				\node[circle, draw=black, fill=gray!10, minimum size=0.4cm, inner sep=0] (outer-\angle) at (\angle:\r + \d) {$u_\label$};
				
				\draw (main-\angle) -- (outer-\angle);
			}
			
		\end{tikzpicture}
	\end{center}
	
	\[
	\mathcal{X} = \left\{ \{u_1, u_2, u_3, u_4\}, \{v_1, v_2\}, \{v_3, v_4\} \right\}
	\]
	
	\noindent{Case $n \geq 5$:}
	
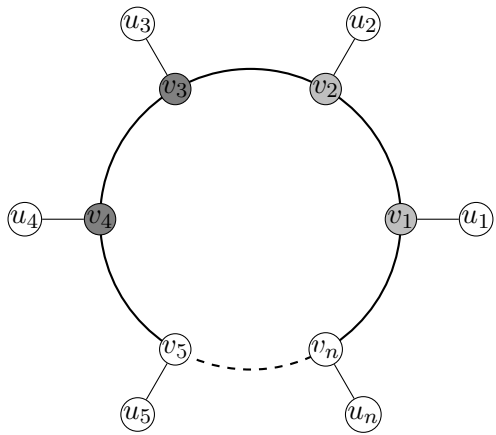
\begin{center}
		\begin{tikzpicture}[scale=1]
			
			\pgfmathsetmacro{\r}{2}
			\pgfmathsetmacro{\d}{1}
			
			\draw[thick] (0,0) circle (\r);
			
			\foreach \angle/\label/\col in { 0/1/gray!50, 60/2/gray!50, 120/3/gray, 180/4/gray, 240/5/white, 300/n/white}
			{
				\node[circle, draw=black, fill=\col, minimum size=0.4cm, inner sep=0] (main-\angle) at (\angle:\r) {$v_\label$};
				
				\node[circle, draw=black, fill=white, minimum size=0.4cm, inner sep=0] (outer-\angle) at (\angle:\r + \d) {$u_\label$};
				
				\draw (main-\angle) -- (outer-\angle);
			}
			
			\draw[dashed, white, line width=2pt] (240:\r) arc (240:300:\r);
			
		\end{tikzpicture}
	\end{center}
	
	\[
	\mathcal{X} = \left\{ \{u_1, u_2, \dots, u_n, v_5, \dots, v_n\}, \{v_1, v_2\}, \{v_3, v_4\} \right\}.
	\]
	It can be verified that these partitions satisfy the conditions of a $2$-fair coalition, thus proving that $C_{2f}(P_n \circ K_1) = 3$ for $n \geq 3$.
	
\end{proof}

\section{Graphs with large $k$-fair coalition number}

Identifying graphs of order $n$ for which the $2$-fair coalition number is either $n$ or $n-1$ is an interesting problem. In this section, we study graphs with a large $2$-fair coalition number.
We need the following theorems:

\begin{theorem}\label{th9}{\rm\cite{Aust}}
	If $C_2(G)=n$, then  $\deg(v)\geq n-2$ for every vertex $v$.
\end{theorem}
	\begin{theorem}\label{bound}{\rm\cite{Aust}}
		For any tree $T$ of order $n$, $C_2(T)\leq \frac{n}{2}+1$.
	\end{theorem}

Similar to Theorem \ref{th9}, it is easy to see that if $C_{2f}(G) = n$ for a graph $G$, then $\deg(v) \geq n-2$ for every vertex $v$. Furthermore, if $v_0$ is a vertex with $\deg(v_0) = n-2$, then $v_0$ can only form a $2$-fair coalition with the single vertex it is not adjacent to. Consequently, the number of vertices of degree $n-2$ must be even and so we have the following corollary.

\begin{corollary}
	If $G$ is an $(n-2)$-regular graph of order $n$ with $C_{2f}(G)=n$, then $n$ is even and $G$ is isomorphic to graph $K_n \setminus M$, where $M$ is a perfect matching of the complete graph $K_n$. 
\end{corollary}

\begin{lemma}\label{lem:4-19}
	For any tree $T$ of order $n$, we have:
	\[
	C_{2f}(T) \leq \left\lfloor \frac{n}{2} \right\rfloor + 1
	\]
\end{lemma}

\begin{proof}
	If follows from Theorem \ref{bound}.\qed
\end{proof}

\begin{corollary}	
	Let $T$ be a tree of order $n$.
	\begin{enumerate}
		\item[(i)] If $C_{2f}(T) = n$, then $T = P_2$.
		\item[(ii)] If $C_{2f}(T) = n - 1$, then $T = P_3$ or $T = P_4$.
	\end{enumerate}
\end{corollary}

\begin{proof}
	\begin{enumerate}
		\item[(i)] 
		By Lemma \ref{lem:4-19}, we have $C_{2f}(T) \leq \left\lfloor \frac{n}{2} \right\rfloor + 1$. If $C_{2f}(T) = n$, then it follows that:
		\[
		n \leq \left\lfloor \frac{n}{2} \right\rfloor + 1.
		\]
		This inequality implies $n \leq 2$. The only tree of order $2$ is the path $P_2$.
		
		\item[(ii)]  If $C_{2f}(T) = n - 1$, then  by Lemma \ref{lem:4-19}, we have $C_{2f}(T) \leq \left\lfloor \frac{n}{2} \right\rfloor + 1$. So
		\[
		n - 1 \leq \left\lfloor \frac{n}{2} \right\rfloor + 1.
		\]
		This inequality implies $n \leq 4$. We now check all trees of order $n \leq 4$:
		\begin{itemize}
			\item For $n = 1, 2$: We have $C_{2f}(T) = n$ (the only possible value).
			\item For $n = 3$: The only tree is the path $P_3$. By Corollary \ref{cor:4-14}, $C_{2f}(P_3) = 2 = 3 - 1$.
			\item For $n = 4$: The trees are the path $P_4$ and the star $K_{1,3}$. By Corollary \ref{cor:4-14}, $C_{2f}(P_4) = 3 = 4 - 1$, while $C_{2f}(K_{1,3}) = 2$.
		\end{itemize}
		Therefore, the only trees achieving $C_{2f}(T) = n - 1$ are $P_3$ and $P_4$.\qed
	\end{enumerate}
\end{proof}

\section{ $2$-fair coalition number of cubic graphs}
In this section, we proceed to study the $2$-fair coalition number of cubic ($3$-regular) graphs. A $3$-regular graph on $4$ vertices is the complete graph $K_4$, for which $C_{2f}(K_4) = 4$. All $3$-regular graphs with six vertices are isomorphic to one of the following two graphs:
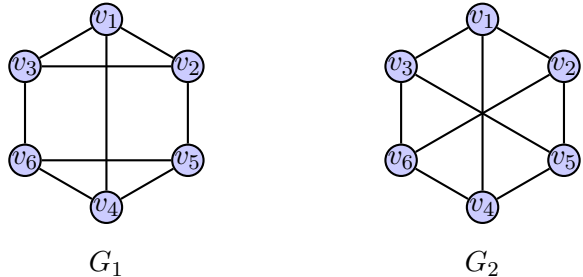
\begin{center}
	\begin{tikzpicture}[scale=0.5,
		node/.style={circle, draw=black, fill=blue!20, minimum size=0.4cm, inner sep=0, thick},
		edge/.style={thick}
		]
		
		\node[node] (v1) at (90:2.5) {$v_1$};
		\node[node] (v2) at (30:2.5) {$v_2$};
		\node[node] (v5) at (-30:2.5) {$v_5$};
		\node[node] (v4) at (-90:2.5) {$v_4$};
		\node[node] (v6) at (-150:2.5) {$v_6$};
		\node[node] (v3) at (150:2.5) {$v_3$};
		
		\draw[edge] (v1) -- (v2);
		\draw[edge] (v2) -- (v3);
		\draw[edge] (v3) -- (v1);
		\draw[edge] (v4) -- (v5);
		\draw[edge] (v5) -- (v6);
		\draw[edge] (v6) -- (v4);
		\draw[edge] (v1) -- (v4);
		\draw[edge] (v2) -- (v5);
		\draw[edge] (v3) -- (v6);
		
		\node[draw=none, fill=none] (label1) at (270:4) {$G_1$};
		
		\begin{scope}[xshift=10cm]
			\node[node] (v1) at (90:2.5) {$v_1$};
			\node[node] (v2) at (30:2.5) {$v_2$};
			\node[node] (v5) at (-30:2.5) {$v_5$};
			\node[node] (v4) at (-90:2.5) {$v_4$};
			\node[node] (v6) at (-150:2.5) {$v_6$};
			\node[node] (v3) at (150:2.5) {$v_3$};
			
			\draw[edge] (v1) -- (v2);
			\draw[edge] (v1) -- (v3);
			\draw[edge] (v3) -- (v5);
			\draw[edge] (v4) -- (v5);
			\draw[edge] (v2) -- (v6);
			\draw[edge] (v6) -- (v4);
			\draw[edge] (v1) -- (v4);
			\draw[edge] (v2) -- (v5);
			\draw[edge] (v3) -- (v6);
			
			\node[draw=none, fill=none] (label2) at (270:4) {$G_2$};
		\end{scope}
	\end{tikzpicture}\label{cubic6}	
\end{center}

\begin{theorem}\label{thm:4-20}
	Let $G$ be a $3$-regular graph of order $n$. If $X$ is a $2$-fair dominating set of $G$ and $|X| = k$, then $k \geq \frac{2n}{5}$.
\end{theorem}

\begin{proof}
	Since the $n - k$ vertices in $V(G)\setminus X$ must be $2$-fair dominated by $X$, there must be at least $2(n - k)$ edges from $X$ to $V(G) \setminus X$. On the other hand, since $G$ is $3$-regular, each vertex in $X$ has at most $3$ incident edges. Therefore, the total number of edges incident to vertices in $X$ is at most $3k$. This yields the inequality:
	\[
	3k \geq 2(n - k)
	\]
	Solving this inequality completes the proof.\qed 
\end{proof}

\begin{theorem}
	For the $3$-regular graphs $G_1$ and $G_2$ of order $6$, the following results hold:
	\begin{enumerate}
		\item[(i)] $C_{2f}(G_1) = 4.$
		\item[(ii)] $C_{2f}(G_2) = 3.$
	\end{enumerate}
\end{theorem}

\begin{proof}
	Since the number of vertices is $6$ and the degree of each vertex is less than $6-2=4$, we have $C_{2f}(G_i) \neq 6$ for $i \in \{1,2\}$.
	
	Now, suppose, for the sake of contradiction, that $\mathcal{X}$ is a $2$-fair coalition partition with $|\mathcal{X}| = 5$. Then $\mathcal{X}$ consists of one set $A$ of size $2$ and four singleton sets. Furthermore, by Theorem \ref{thm:4-20}, any $2$-fair dominating set must have at least $3$ vertices. Therefore, all four singleton sets must form a coalition with $A$.
	\begin{enumerate}
		\item[(i)]
	 Consider an arbitrary pair of vertices in $A$. There exists a vertex $v_0$ that is adjacent to both vertices in $A$. Let $v'$ be the other vertex adjacent to $v_0$. The singleton set $\{v'\}$ cannot form a coalition with $A$ because in $A \cup \{v'\}$, we have $|N[v_0] \cap (A \cup \{v'\})| = 3 \neq 2$. This contradicts the assumption that $|\mathcal{X}| = 5$.
	
	We now present a $2$-fair coalition partition of size $4$:
	\[
	\mathcal{X} = \big\{ \{v_1, v_2\}, \{v_3\}, \{v_4, v_5\}, \{v_6\} \big\}.
	\]
	It can be verified that this is a valid partition, confirming that $C_{2f}(G_1) = 4$.
		\item[(ii)]
	Consider the set $A$ of size $2$. If the two vertices in $A$ are not adjacent, a contradiction arises similarly to part (i). If they are adjacent, one can easily see that $A$ cannot form a coalition with an arbitrary singleton set $\{v'\}$, as the degree of some vertices in $A \cup \{v'\}$ would be $1$, violating the condition for a $2$-fair dominating set. Consequently, the assumption $|\mathcal{X}| = 5$ is impossible.
	
	Now, suppose $|\mathcal{X}| = 4$. If $\mathcal{X}$ contains a set of size $2$, then a singleton set must form a coalition with it, which, as argued above, is impossible. Therefore, $\mathcal{X}$ must consist of one set $B$ of size $3$ and three singleton sets. The vertices of $B$ must form a path (e.g., $v_1, v_3, v_5$). Then, the vertex $v_6$, which is adjacent to the middle vertex $v_3$, cannot form a coalition with $B$ because the degree of the remaining vertices in their union would be $3$, not $2$. Note that if the three vertices do not form a $P_3$, they would themselves constitute a $2$-fair dominating set, which contradicts the requirement for forming a coalition within a partition.
	
	We now present a $2$-fair coalition partition of size $3$:
	\[
	\mathcal{X} = \big\{ \{v_1, v_4\}, \{v_3, v_5\}, \{v_2, v_6\} \big\}.
	\]
	It can be verified that this is a valid partition, confirming that $C_{2f}(G_2) = 3$.\qed
	\end{enumerate} 
\end{proof}

\section{Conclusion}

	This work introduces the concept of $k$-fair coalition in graphs and explores fundamental properties of the $k$-fair coalition number. We establish that every graph $G$ admits a $k$-fair coalition partition and derive bounds for $C_{kf}(G)$. Applying these bounds allows us to compute exact $k$-fair coalition values for several well-structured graphs. We further characterize graphs exhibiting extreme $k$-fair coalition values.
	
	We propose the following open questions and promising research directions related to the $k$-fair coalition number:
	
	\begin{enumerate}
		\item[(i)] Determine exact values of the $k$-fair coalition number for fundamental graph classes (paths, cycles, trees, bipartite graphs) for various $k$.
		
		\item[(ii)] Establish Nordhaus–Gaddum type bounds on the sum and product of $k$-fair coalition numbers of a graph $G$ and its complement $\overline{G}$.
		
		\item[(iii)] Investigate the behavior of the $k$-fair coalition number under standard graph operations (corona, Cartesian product, join, lexicographic product).
		
		\item[(iv)] For any $k$-fair coalition partition $\Psi$ of $G$, define the corresponding $k$-fair coalition graph $kFCG(G, \Psi)$ with vertices representing partition blocks and edges indicating $k$-fair coalition relationships. Studying properties of these graphs offers rich opportunities for further investigation.
	\end{enumerate}

\end{document}